\documentclass[10pt, reqno]{amsart}
\usepackage{cite}
\usepackage{amsmath}
\usepackage{amsfonts}
\usepackage{amssymb}
\usepackage{amsthm}
\usepackage[sc]{mathpazo}
\usepackage[T1]{fontenc} 
\usepackage[dvipsnames]{xcolor}
\usepackage{graphicx}
\usepackage{hyperref}

\usepackage{enumitem}

\theoremstyle{plain}
\newtheorem{theorem}{Theorem}
\numberwithin{theorem}{section}

\newtheorem{lemma}[theorem]{Lemma}
\newtheorem{corollary}[theorem]{Corollary}

\theoremstyle{definition}

\newcommand{\C}{\mathbb{C}}
\newcommand{\F}{\mathbb{F}}
\newcommand{\M}{\mathsf{M}}
\newcommand{\R}{\mathbb{R}}
\newcommand{\T}{\mathsf{T}}

\newcommand{\colspace}{\operatorname{col}}
\newcommand{\nullspace}{\operatorname{null}}
\newcommand{\rank}{\operatorname{rank}}
\renewcommand{\vec}[1]{\mathbf{#1}}
\newcommand{\norm}[1]{\| #1 \|}

\newcommand{\tworowvector}[2]{[#1\,\,#2]}

\allowdisplaybreaks

\begin{document}
\renewcommand{\theenumi}{\alph{enumi}} 
\setenumerate{itemsep=3pt}

\title{The linear targeting problem}

\author[K.~Bierly]{Kyle Bierly}
\address{Department of Mathematics and Statistics, Pomona College, 610 N. College Ave., Claremont, CA 91711, USA}
\email{kebierly@gmail.com}

\author[S.R.~Garcia]{Stephan Ramon Garcia}
\address{Department of Mathematics and Statistics, Pomona College, 610 N. College Ave., Claremont, CA 91711, USA}
\email{stephan.garcia@pomona.edu}
\urladdr{\url{https://stephangarcia.sites.pomona.edu/}}

\author[R.A.~Horn]{Roger A. Horn}
\address{Tampa, Florida, USA}
\email{rhorn@math.utah.edu}

\thanks{SRG partially supported by NSF grant DMS-2054002.}

\begin{abstract}
For given real or complex $m \times n$ data matrices $X$, $Y$, we investigate when there is a matrix $A$ such that $AX = Y$, and $A$ is invertible, Hermitian, positive (semi)definite, unitary, an orthogonal projection, a reflection, complex symmetric, or normal.
\end{abstract}


\maketitle

\section{Introduction}

The \emph{linear targeting problem} is to construct an $A\in \M_{m}(\F)$ such that $AX=Y$, in which $X,Y\in \M_{m\times n}(\F)$ are given
and $\F=\R$ or $\C$. We say that $X$ is the \emph{source}, $Y$ is the \emph{target}, and $A$ is a \emph{targeting matrix}. 
If there is an $A\in \M_{m}(\F)$ with property $\mathcal{P}$ such that $AX=Y$, we say that $A$ solves the 
\emph{$\mathcal{P}$-targeting problem for $X$ and $Y$}. In this paper, we consider the following
properties: invertible, Hermitian, positive (semi)definite, unitary,
orthogonal projection, reflection, complex symmetric, and normal.

If $AX=Y$, then $\nullspace  X\subseteq \nullspace Y$, so this condition (or something that implies it) is found in all of our results. The invertible
linear targeting problem is just the row equivalence problem, which has a solution if and only if $\nullspace X=\nullspace Y$ \cite[Theorem 4.4.1]{GH2}.

If $m<n$ and $AX=Y$, then after a simultaneous permutation of the columns of 
$X$ and $Y$ we may assume that $X=\tworowvector{X_1}{X_2}$, $Y = \tworowvector{Y_1}{Y_2}$, $X_{1},Y_{1}\in \M_{m\times k}(\F)$, $k\leq m$, 
and $\colspace X_{2}\subseteq \colspace X_{1}$. For example, we could choose $k=\rank X$, but any $k$ between $\rank X$ and $m$ will do. Since there is
some $B\in \M_{k\times (n-k)}(\F)$ such that $X_{2}=X_{1}B$ \cite[Theorem 1.6.21.a]{GH2}, we have $AX_{1}=Y_{1}$, from which it follows that 
$AX_{2}=A( X_{1}B) =Y_{1}B=Y_{2}$. Thus, in considering the linear targeting problem $AX=Y$ with $X,Y\in \M_{m\times n}(\F)$, it
suffices to consider only the case $m\geq n\geq 1$.

For any $Z\in \M_{m\times n}(\F)$ ($m\times n$ matrices over the field $\R$ or $\C$), $Z^*$ is the \emph{conjugate transpose} of $Z$
and $Z^{\dagger }$ is its \emph{Moore--Penrose pseudoinverse}; see \cite[\S\ 17.4]{GH2} or \cite[7.3.P7]{MA2}. The \emph{Euclidean norm} of $\vec{x}
\in \F^{n}$ is denoted by $\norm{ \vec{x}} =( \vec{x}^*\vec{x})^{1/2}$. If $A\in \M_{n}(\F)$, we
write $A\geq 0$ (respectively, $A>0$) if $A\in \M_{n}(\F)$ is \emph{positive semidefinite} (respectively, \emph{positive definite}), that is, $A$
is Hermitian and $\vec{u}^*A\vec{u}\geq 0$ (respectively, $\vec{u}^*A\vec{u}>0)$ for all nonzero $\vec{u}\in \F^{n}$. 
If $B,S\in \M_{m}(\F)$ and $S$ is invertible, then $S^*BS$ is $^*$\emph{congruent} to $B$. If $B$ is, respectively, invertible,
Hermitian, or positive semidefinite, then $S^*BS$ has the same respective properties.

\section{Unconstrained linear targeting}

A linear targeting problem need not have a solution. For example, $AX=Y$ is
impossible if
\begin{equation*}
X= 
\begin{bmatrix}
1 & 0 \\ 
0 & 0
\end{bmatrix}
 \quad \text{and}\quad Y= 
\begin{bmatrix}
0 & 1 \\ 
0 & 0
\end{bmatrix}.
\end{equation*}
Since $\nullspace X\subseteq \nullspace ( AX) $, it is necessary that $\nullspace X\subseteq \nullspace Y$. A computation with pseudoinverses
shows that this condition permits us to identify all solutions of an unconstrained linear targeting problem.

\begin{theorem}
Let $X,Y\in \M_{m\times n}(\F)$ with $m\geq n\geq 1$ and $\nullspace  X\subseteq \nullspace Y$. 
\begin{enumerate}[leftmargin=*]
    \item $( YX^{\dagger }) X=Y$. 
    \item If $Z\in \M_{m}(\F)$ and $A=YX^{\dagger }+Z(I-XX^{\dagger })$,
    then $AX=Y$. 
    \item If $A\in \M_{m}(\F)$ and $AX=Y$, then there is a $Z\in \M_{m}(\F)$ such that $A=YX^{\dagger }+Z(I-XX^{\dagger })$.
\end{enumerate}
\end{theorem}

\begin{proof}
(a) The hypothesis $\nullspace X\subseteq \nullspace Y$ is equivalent to the
condition $\colspace Y^*\subseteq \colspace X^*$. Since $X^{\dagger }X$ is the orthogonal projection (necessarily Hermitian) onto 
$\colspace X^*$, we have $X^{\dagger }XY^*=Y^*$, and hence $Y=YX^{\dagger }X=( YX^{\dagger }) X$. 

\medskip\noindent(b) Compute
\begin{align*}
    AX 
    &= ( YX^{\dagger }+Z(I-XX^{\dagger })) X \\
    &= YX^{\dagger }X+Z(X-XX^{\dagger }X) \\
    &= Y+Z(X-X)=Y.
\end{align*}

\medskip\noindent(c) If $AX=Y$, then
\begin{align*}
    A-YX^{\dagger } 
    &= ( A-YX^{\dagger }) ( I-XX^{\dagger}+XX^{\dagger }) \\
    &= ( A-YX^{\dagger }) ( I-XX^{\dagger }) +(A-YX^{\dagger }) XX^{\dagger } \\
    &= ( A-YX^{\dagger }) ( I-XX^{\dagger }) +(AX-( YX^{\dagger }) X) X^{\dagger } \\
    &= ( A-YX^{\dagger }) ( I-XX^{\dagger }) +(Y-Y) X^{\dagger } \\
    &= ( A-YX^{\dagger }) ( I-XX^{\dagger }) =Z(I-XX^{\dagger }) ,
\end{align*}
in which $Z=A-YX^{\dagger }$.
\end{proof}

The following theorem re-examines the unconstrained linear targeting problem
in the context of a singular value decomposition of the data matrix $X$. 
Later, we find it convenient to employ a singular value decomposition of $Y$, with analogous notation.
If $X=0$, the linear targeting problem has a solution if and only if $Y=0$, so this case is not interesting.

\begin{theorem}\label{Theorem:Basic}
Let $X,Y\in \M_{m\times n}(\F)$ with $m\geq n\geq 1$ and $X\neq 0$. 
There is an $A\in \M_{m}(\F)$ such that $AX=Y$ if and only if $\nullspace X\subseteq \nullspace Y$.
\end{theorem}

\begin{proof}
($\Rightarrow $) If $\vec{u}\in \nullspace X$, then $Y\vec{u}=AX \vec{u}=A\vec{0}=\vec{0}$.

\medskip\noindent($\Leftarrow $) Let $\rank X=r\geq 1$ and let $X=V\Sigma W^*$ be
a singular value decomposition, in which $V\in \M_{m}(\F)$ and $W\in\M_{n}(\F)$ are unitary,
\begin{equation}\label{eq:Sigma}
    \Sigma =
    \begin{bmatrix}
    \Sigma_{r} & 0_{r\times (n-r)} \\ 
    0_{(m-r)\times r} & 0_{(m-r)\times (n-r)}
    \end{bmatrix}
     \in \M_{m\times n}(\R),  
\end{equation}
$r\geq 1$, and $\Sigma_{r}\in \M_{r}(\R)$ is diagonal, invertible,
and positive definite. Partition $W=[W_{1}~W_{2}]$ and $V=[V_{1}~V_{2}]$, in
which $W_{1}\in \M_{n\times r}(\F)$ and $V_{1}\in \M_{m\times r}(\F)$. 
If $X$ has full rank, the block column of zeros in \eqref{eq:Sigma} is absent, 
$W=W_{1}$, and $W_{2}$ is absent. If $r<n$, the columns of 
$W_{2}$ are an orthonormal basis for $\nullspace X$ and the columns of $W_{1}$
are an orthonormal basis for $(\nullspace X)^{\perp }=\colspace X^*$.
Since $\nullspace X\subseteq \nullspace Y$, we have $YW_{2}=0$.

If there is a $B\in \M_{m}$ such that 
\begin{equation}\label{eq:BasicSigma}
    B\Sigma =V^*YW,  
\end{equation}
then $YW=VB\Sigma =(VBV^*)( V\Sigma W^*) W=(VBV^*) XW$ and hence 
\begin{equation*}
    ( VBV^*) X=Y.
\end{equation*}
Thus, $A=VBV^*$ solves the unconstrained targeting problem for $X$ and  $Y$. 
We can partition the right side of \eqref{eq:BasicSigma} in several ways:
\begin{align}
    V^*YW 
    &= V^*Y[W_{1}~W_{2}]=V^*[YW_{1}~YW_{2}]=[V^*YW_{1}~0]  \label{eq:V*YW} \\
    &= 
    \begin{bmatrix}
    V_{1}^*YW_{1} & 0 \\ 
    V_{2}^*YW_{1} & 0
    \end{bmatrix}
     = 
    \begin{bmatrix}
    Z_{1} & 0 \\ 
    Z_{2} & 0
    \end{bmatrix}
     = 
    \begin{bmatrix}
    Z & 0
    \end{bmatrix}
     \in \M_{m\times n}(\F),  \notag
\end{align}
in which the three matrices $Z=V^*YW_{1}\in \M_{m\times r}(\F)$,
$Z_{1}=V_{1}^*YW_{1}\in \M_{r}(\F)$, and $Z_{2}=V_{2}^{*}YW_{1}\in \M_{(m-r)\times r}(\F)$ are determined by the data $X$ and 
$Y$. Define $B_{1}=Z\Sigma_{r}^{-1}\in \M_{m\times r}(\F)$ and let 
\begin{equation}\label{eq:BasicB}
    B= 
    \tworowvector{B_1}{B_2}
     = 
    \begin{bmatrix}
    Z\Sigma_{r}^{-1} & B_{2}
    \end{bmatrix}
     = 
    \begin{bmatrix}
    Z_{1}\Sigma_{r}^{-1} & R \\ 
    Z_{2}\Sigma_{r}^{-1} & S
    \end{bmatrix},
\end{equation}
in which $R\in \M_{r\times (n-r)}(\F)$ and $S\in \M_{n-r}(\F)$ are arbitrary. Then 
\begin{equation*}
    B\Sigma = 
    \begin{bmatrix}
    Z\Sigma_{r}^{-1} & B_{2}
    \end{bmatrix}
    \begin{bmatrix}
    \Sigma_{r} & 0 \\ 
    0 & 0
    \end{bmatrix}
     = 
    \begin{bmatrix}
    Z & 0
    \end{bmatrix}
     =V^*YW,
\end{equation*}
so \eqref{eq:BasicSigma} is satisfied for any choice of $B_{2}$.
\end{proof}

If a property $\mathcal{P}$ is invariant under unitary similarity, the preceding theorem shows that the $\mathcal{P}$-targeting problem is
equivalent to a matrix completion problem. The source and target matrices determine $B_{1}$; if we can choose 
$B_{2}$ so that $B$ (and hence also $VBV^*)$ has property $\mathcal{P}$, then the $\mathcal{P}$-targeting
problem has a solution. For example, if $\mathcal{P}$ is the property that $A$ is invertible, then $B$ must be invertible and hence the data must
ensure that $B_{1}$ has full rank. The following corollary reveals a familiar criterion for row equivalence of two matrices; see \cite[Theorem 4.4.1 and P.8.30]{GH2}.

\begin{corollary}\label{Corollary:Invertible}
Let $X,Y\in \M_{m\times n}(\F)$ with $m\geq n\geq 1$
and $X\neq 0$. There is an invertible $A\in \M_{m}(\F)$ such that $AX=Y$ if and only if $\nullspace X=\nullspace Y$.
\end{corollary}

\begin{proof}
($\Rightarrow $) If $Y=AX$ and $A$ is invertible, then $X=A^{-1}Y$.
Consequently, $\nullspace X\subseteq \nullspace Y$ and $\nullspace Y\subseteq \nullspace X$, so $\nullspace X=\nullspace Y$.

\medskip\noindent($\Leftarrow $) Let $r=\rank X\geq 1$. There is a choice of $B_{2}$ in \eqref{eq:BasicB} 
that makes $B$ invertible if and only if $\rank B_{1}=r$.
If $\nullspace X=\nullspace Y$, then $\rank X=\rank Y$, and the relations
\begin{equation}\label{eq:rankZ}
    \rank Y=\rank ( V^*YW) =\rank  
    \tworowvector{Z}{0}
     =\rank Z=\rank ( Z\Sigma_{r}^{-1}) =\rank  B_{1}  
\end{equation}
ensure that $\rank B_{1}=r$.
\end{proof}

\section{Two Lemmas}
The following lemma memorializes
some consequences of the basic assumption $\nullspace X\subseteq \nullspace  Y$.

\begin{lemma}
    Let $X,Y\in \M_{m\times n}(\F)$ with $m\geq n\geq 1$ and $X\neq 0$.
    Suppose that $\nullspace X\subseteq \nullspace Y$ and adopt the notation used in the proof of Theorem \ref{Theorem:Basic}. 
    Then $X^*Y$ is unitarily
    similar to $\Sigma_{r}Z_{1}\oplus 0_{n-r}$, in which $Z_{1}=V_{1}^{*}YW_{1}$.
\end{lemma}

\begin{proof}\renewcommand{\qedsymbol}{}
    Compute 
    \begin{align}
        W^*( X^*Y) W 
        &= W^*( W\Sigma V^{*}) YW=\Sigma ( V^*YW) =\Sigma  
        \tworowvector{Z}{0}\notag \\
        &=  
        \begin{bmatrix}
        \Sigma_{r} & 0 \\ 
        0 & 0
        \end{bmatrix}
        \begin{bmatrix}
        Z_{1} & 0 \\ 
        Z_{2} & 0
        \end{bmatrix}
         = 
        \begin{bmatrix}
        \Sigma_{r}Z_{1} & 0 \\ 
        0 & 0
        \end{bmatrix}
         .  \rlap{$\qquad \Box$} \label{eq:BasicX*Y}
    \end{align}
\end{proof}

In our computations with block matrices, we require the three types of Schur complements in the next lemma.

\begin{lemma}\label{Lemma:SchurComplement}
    Let $H\in \M_{r}(\F)$, $L\in \M_{(m-r)\times r}(\F)$, and $\lambda \in \F$. Let 
    \begin{equation*}
        B= 
        \begin{bmatrix}
        H & L^* \\ 
        L & \lambda I_{m-r}
        \end{bmatrix}
         \in \M_{m}(\F).
    \end{equation*}
    \begin{enumerate}[leftmargin=*]
        \item If $\lambda \in \R\backslash \{0\}$, then $B$ is $^*$congruent to $( H-\lambda^{-1}L^*L) \oplus ( \lambda I_{m-r}) $. 
        \item If $H$ is Hermitian and invertible, then $B$ is $^*$congruent to $H\oplus ( \lambda I_{m-r}-LH^{-1}L^*) $. 
        \item If $H$ is Hermitian and $\nullspace H\subseteq \nullspace L$, then $B$ is $^*$congruent to $H\oplus ( \lambda I_{m-r}-LH^{\dagger}L^*) $.
    \end{enumerate}
\end{lemma}

\begin{proof}
(a) Compute the $^*$congruence
\begin{equation*}
    \begin{bmatrix}
        I_{r} & -\lambda^{-1}L^* \\ 
        0 & I_{m-r}
    \end{bmatrix}  
    \begin{bmatrix}
        H & L^* \\ 
        L & \lambda I_{m-r}
    \end{bmatrix}  
    \begin{bmatrix}
        I_{r} & 0 \\ 
        -\lambda^{-1}L & I_{m-r}
    \end{bmatrix}
     = 
    \begin{bmatrix}
        H-\lambda^{-1}L^*L & 0 \\ 
        0 & \lambda I_{m-r}
    \end{bmatrix}.
\end{equation*}

\medskip\noindent(b) The product
\begin{equation*}
    \begin{bmatrix}
        I_{r} & 0 \\ 
        -LH^{-1} & I_{m-r}
    \end{bmatrix}  
    \begin{bmatrix}
        H & L^* \\ 
        L & \lambda I_{m-r}
    \end{bmatrix}  
    \begin{bmatrix}
        I_{r} & -H^{-1}L^* \\ 
        0 & I_{m-r}
    \end{bmatrix}
     = 
    \begin{bmatrix}
        H & 0 \\ 
        0 & \lambda I_{m-r}-LH^{-1}L^*
    \end{bmatrix}
\end{equation*}
is a $^*$congruence since $H$ is Hermitian. 

\medskip\noindent(c) Since $I-H^{\dagger }H$ is the orthogonal projection onto $\nullspace H$
and $\nullspace H\subseteq \nullspace L$, we have $L(I_{m-r}-H^{\dagger}H)=0 $ and the product 
\begin{equation*}
    \begin{bmatrix}
        I_{r} & 0 \\ 
        -LH^{\dagger } & I_{m-r}
    \end{bmatrix}  
    \begin{bmatrix}
        H & L^* \\ 
        L & \lambda I_{m-r}
    \end{bmatrix}  
    \begin{bmatrix}
        I_{r} & -H^{\dagger }L^* \\ 
        0 & I_{m-r}
    \end{bmatrix}
     = 
    \begin{bmatrix}
        H & 0 \\ 
        0 & \lambda I_{m-r}-LH^{\dagger }L^*
    \end{bmatrix}
\end{equation*}
is a $^*$congruence.
\end{proof}

\section{Hermitian targeting}

\begin{theorem}\label{Theorem:Hermitian}
    Let $X,Y\in \M_{m\times n}(\F)$ with $m\geq n\geq 1$
    and $X\neq 0$. There is a Hermitian $A\in \M_{m}(\F)$ such that $AX=Y$
    if and only if $\nullspace X\subseteq \nullspace Y$ and $X^*Y$ is Hermitian.
\end{theorem}

\begin{proof}
    ($\Rightarrow $) If $AX=Y$ and $A$ is Hermitian, then $\nullspace X\subseteq  \nullspace Y$ and 
    \begin{equation*}
        ( X^*Y)^*=( X^*AX)^*=X^{*}AX=X^*Y.
    \end{equation*}
    
    \medskip\noindent($\Leftarrow $) Let $r=\rank X$. It suffices to show that there is a
    choice of $B_{2}$ in \eqref{eq:BasicB} such that $B$ is Hermitian. Let
    \begin{equation}\label{eq:HermitianB}
        B= 
        \begin{bmatrix}
        Z_{1}\Sigma_{r}^{-1} & ( Z_{2}\Sigma_{r}^{-1})^* \\ 
        Z_{2}\Sigma_{r}^{-1} & \lambda I_{m-r}
        \end{bmatrix}
         = 
        \begin{bmatrix}
        B_{1} & B_{2}
        \end{bmatrix}
         ,  
    \end{equation}
    in which $\lambda \in \R$. Since $X^{*}Y$ is Hermitian, \eqref{eq:BasicX*Y} ensures that $\Sigma_{r}Z_{1}$ is
    Hermitian, which implies that $\Sigma_{r}^{-* }( \Sigma_{r}Z_{1}) \Sigma_{r}^{-1}=Z_{1}\Sigma_{r}^{-1}$ is Hermitian and
    hence $B$ is Hermitian. 
\end{proof}

\begin{corollary}\label{Corollary:InvertibleHermitian}
Let $X,Y\in \M_{m\times n}(\F)$ with $m\geq n\geq 1$ and $X\neq 0$. There is an invertible Hermitian 
$A\in \M_{m}(\F)$ such that $AX=Y$ if and only if $\nullspace X=\nullspace Y$ and $X^*Y$ is Hermitian.
\end{corollary}

\begin{proof}
    ($\Rightarrow $) It follows from Theorem \ref{Theorem:Hermitian} that $X^*Y$
    is Hermitian and $\nullspace X\subseteq \nullspace Y$. Since $X=A^{-1}Y$, it
    follows that $\nullspace Y\subseteq \nullspace X$, so $\nullspace X=\nullspace Y$.

    \medskip\noindent($\Leftarrow $) Let $r=\rank X$. 
    The rank-nullity theorem ensures that $\rank Y=r$ since $\dim \nullspace X=\dim \nullspace Y$. Let
    \begin{equation*}
        B= 
        \tworowvector{B_1}{B_2}
         = 
        \begin{bmatrix}
            Z_{1}\Sigma_{r}^{-1} & ( Z_{2}\Sigma_{r}^{-1})^* \\ 
            Z_{2}\Sigma_{r}^{-1} & \lambda I_{m-r}
        \end{bmatrix}
         = 
        \begin{bmatrix}
            H & L^* \\ 
            L & \lambda I_{m-r}
        \end{bmatrix}
    \end{equation*}
    be the Hermitian matrix defined in \eqref{eq:HermitianB}, in which $H=Z_{1}\Sigma_{r}^{-1}\in \M_{r}(\F)$ is Hermitian, 
    $L=Z_{2}\Sigma_{r}^{-1}\in \M_{(m-r)\times r}(\F)$, $B_{1}\in \M_{m\times r}(\F)$, and $\lambda \in \R$ is nonzero. 
    Then \eqref{eq:rankZ} ensures that $B_{1}$ has full column rank, and hence 
    \begin{equation}\label{eq:Intersection1}
        \nullspace H\cap \nullspace ( L^*L) =\nullspace H\cap \nullspace L=\{\vec{0}\}.  
    \end{equation}
    Lemma \ref{Lemma:SchurComplement}.a implies that
    \begin{equation*}
        \det B=( \det ( \lambda I_{m-r}) ) ( \det (H-\lambda^{-1}L^*L) ) =\lambda^{m-r}\det (H-\lambda^{-1}L^*L) .
    \end{equation*}
    We need to show that there is some real nonzero $\lambda $ such that $\det( H-\lambda^{-1}L^*L) \neq 0$. Consider the analytic
    function $f(z)=\det ( H-z^{-1}L^*L) $ on $\C \backslash \{0\}$. If $f(\lambda )=0$ for every real $\lambda \neq 0$, then 
    $f(z)=0$ for every $z\in \C\backslash \{0\}$. To show that $f(\lambda)\neq 0$ for some real nonzero $\lambda $, 
    it suffices to show that $f(i)\neq 0$. If $f(i)=0$, there is a nonzero $\vec{u}\in \C^{r}$ such that $( H+iL^*L) \vec{u}=\vec{0}$. However,
    \begin{align*}
        ( H+iL^*L) \vec{u}=\vec{0} 
        &\implies \vec{u}^* H \vec{u}+i \vec{u}^* L^*L\vec{u}=\vec{u}^*( H+iL^*L) \vec{u}=0 \\
        &\implies \operatorname{Im}( \vec{u}^*H\vec{u}+i\vec{u}^*L^*L\vec{u}) =0 \\
        &\implies ( L\vec{u})^*( L\vec{u}) =\vec{u}^*L^*L\vec{u}=0 \\
        &\implies L\vec{u}=\vec{0} \\
        &\implies H\vec{u}=( H+iL^*L) \vec{u}=\vec{0} \\
        &\implies \nullspace H\cap \nullspace L\neq \{\vec{0}\},
    \end{align*}
    which contradicts \eqref{eq:Intersection1}. We conclude that $B$ is invertible
    and Hermitian for some real nonzero $\lambda$.
\end{proof}

If a nonzero $Y\in \M_{m\times n}(\F)$ is given, what are
all the $X\in \M_{m\times n}(\F)$ such that $Y=AX$ for some Hermitian $A?$
To analyze this question we start with a singular value decomposition of $Y$ instead of $X$.

\begin{corollary}\label{Corollary:HermitianZ}
    Let $X,Y\in \M_{m\times n}(\F)$ with $m\geq n\geq 1$ and $r=\rank Y\geq 1$. Let $Y=V\Sigma W^*$ be a singular value
    decomposition, in which 
    \begin{equation*}
        V^*YW=\Sigma = 
        \begin{bmatrix}
            \Sigma_{r} & 0 \\ 
            0 & 0
        \end{bmatrix}
         \in \M_{m\times n}(\R)
    \end{equation*}
    and $\Sigma_{r}\in \M_{r}(\R)$ is diagonal and positive definite.
    Let $Z = V^*XW$ and partition
    \begin{equation*}
        Z = 
        \begin{bmatrix}
            Z_{11} & Z_{12} \\ 
            Z_{21} & Z_{22}
        \end{bmatrix}
    \end{equation*}
    conformally to $\Sigma $.  Let $Z_1 = \Big[\begin{smallmatrix}Z_{11} \\ Z_{21} \end{smallmatrix} \Big]\in \M_{m \times r}(\F)$.
    \begin{enumerate}[leftmargin=*]
        \item $X^*Y=Y^*X$ if and only if $Z_{12} = 0$ and $\Sigma_{r}Z_{11}=Z_{11}^*\Sigma_{r}$.

        \item If $Z_{12} = 0$, $\Sigma_{r}Z_{11}=Z_{11}^*\Sigma_{r}$, $\rank Z_1 = r$, and 
        $( Z_{21}\nullspace Z_{11}) \cap \colspace Z_{22}=\{\vec{0} \}$, then there is a Hermitian $A \in \M_m(\F)$ such that $AX = Y$.

        \item If $Z_{12} = 0$, $\Sigma_{r}Z_{11}=Z_{11}^*\Sigma_{r}$, and $Z_{11}$ is invertible, then for each choice of $Z_{21} \in \M_{(m-r)\times r}(\F)$ and $Z_{22} \in \M_{(m-r) \times (n-r)}(\F)$ there is a Hermitian $A \in \M_m(\F)$ such that $AX = Y$.
    \end{enumerate}
\end{corollary}

\begin{proof}
    (a) A computation reveals that $X^*Y$ is Hermitian if and only if
    \begin{align*}
        X^*Y=Y^*X 
        &\iff  X^*V\Sigma W^*=W\Sigma^{\T}V^*X \\
        &\iff ( W^*X^*V) \Sigma =\Sigma^{\T}( V^*XW) \\
        &\iff  
        \begin{bmatrix}
            Z_{11}^* & Z_{21}^* \\ 
            Z_{12}^* & Z_{22}^*
        \end{bmatrix}
        \begin{bmatrix}
            \Sigma_{r} & 0 \\ 
            0 & 0
        \end{bmatrix}
         = 
        \begin{bmatrix}
            \Sigma_{r} & 0 \\ 
            0 & 0
        \end{bmatrix}  
        \begin{bmatrix}
            Z_{11} & Z_{12} \\ 
            Z_{21} & Z_{22}
        \end{bmatrix}
         \\
        &\iff  
        \begin{bmatrix}
            Z_{11}^*\Sigma_{r} & 0 \\ 
            Z_{12}^*\Sigma_{r} & 0
        \end{bmatrix}
         = 
        \begin{bmatrix}
            \Sigma_{r}Z_{11} & \Sigma_{r}Z_{12} \\ 
            0 & 0
        \end{bmatrix}.
    \end{align*}
    Therefore, $X^*Y$ is Hermitian if and only if $Z_{12}=0$ and $\Sigma_{r}Z_{11}=Z_{11}^*\Sigma_{r}$. 
    
    \medskip\noindent(b) We claim that $\nullspace X\subseteq \nullspace Y$ if and only if $\nullspace Z\subseteq \nullspace \Sigma $. 
    To validate this claim, suppose that $\nullspace X\subseteq \nullspace Y$. Then
    \begin{align*}
        \vec{u}\in \nullspace Z 
        &\implies  \vec{0}=Z\vec{u}=(V^*XW) \vec{u}=\vec{0} \\
        &\implies X( W\vec{u}) =\vec{0} \\
        &\implies Y( W\vec{u}) =\vec{0} \\
        &\implies V^*Y( W\vec{u}) =\vec{0} \\
        &\implies \vec{0}=( V^*YW) \vec{u}=\Sigma \vec{u} \\
        &\implies \vec{u}\in \nullspace \Sigma ,
    \end{align*}
    from which we conclude that $\nullspace Z\subseteq \nullspace \Sigma $. The
    converse follows from a parallel argument. 
    Now assume that $Z_{12} = 0$ and $\Sigma_r Z_{11} = Z_{11}^* \Sigma_r$, which ensures that $X^*Y = Y^*X$.
    Let 
    $\vec{u}= 
    \big[
    \begin{smallmatrix}
    \vec{u}_{1} \\ 
    \vec{u}_{2}
    \end{smallmatrix}
    \big] \in \F^{n}$, in which $\vec{u}_{1}\in \F^{r}$.
    Then $\vec{u}\in \nullspace \Sigma $ if and only if $\vec{u}_{1}=\vec{0}$. 
    If $\rank Z_1= r$ and there is a $\vec{u} \in \nullspace Z$ such that $\vec{u}_1 \neq \vec{0}$,
    then
    \begin{equation*}
        \vec{0}=Z\vec{u}= 
        \begin{bmatrix}
        Z_{11} & 0 \\ 
        Z_{21} & Z_{22}
        \end{bmatrix}
        \begin{bmatrix}
        \vec{u}_{1} \\ 
        \vec{u}_{2}
        \end{bmatrix}
         = 
        \begin{bmatrix}
        Z_{11}\vec{u}_{1} \\ 
        Z_{21}\vec{u}_{1}+Z_{22}\vec{u}_{2}
        \end{bmatrix}
         = 
        \begin{bmatrix}
        \vec{0} \\ 
        Z_{21}\vec{u}_{1}+Z_{22}\vec{u}_{2}
        \end{bmatrix}.
    \end{equation*}
    Since $\rank Z_1 = r$, it follows that $Z_{21} \vec{u}_1 \neq \vec{0}$,
    which implies that $Z_{22} \vec{u}_2 \neq\vec{0}$ and hence
    $( Z_{21}\nullspace Z_{11}) \cap \colspace Z_{22}\neq \{\vec{0}\}$.
    We conclude that $\nullspace Z \subseteq \nullspace \Sigma$. 
    
    \medskip\noindent(c) If $Z_{11}$ is invertible, then $\rank Z_1 = r$ and $\nullspace Z_{11}=\{\vec{0}\}$, so
    $( Z_{21}\nullspace Z_{11}) \cap \colspace Z_{22} \subseteq Z_{21} \nullspace Z_{11} =\{\vec{0}\}$.
\end{proof}

\section{Positive semidefinite targeting}

We now investigate the positive semidefinite and positive definite targeting
problems. An important property of positive semidefinite matrices $A$ is
that $A\vec{u}=\vec{0}$ if and only if $\vec{u}^*A\vec{u}=0$; see \cite[Corollary 15.1.18]{GH2} or \cite[Observation 7.1.6]{MA2}.

\begin{theorem}\label{Theorem:PSD}
Let $X,Y\in \M_{m\times n}(\F)$ with $m\geq n\geq 1$ and $X\neq 0$. 
There is a positive semidefinite $A\in \M_{m}(\F)$ such that $AX=Y$ if and only if $X^*Y\geq 0$ and $\nullspace Y=\nullspace (X^*Y)$.
\end{theorem}

\begin{proof}
    ($\Rightarrow $) If $AX=Y$ and $A\geq 0$, then $X^*Y=X^*AX$,
    which is positive semidefinite. To verify the assertion about null spaces, use the fact that $X^*Y\geq 0$ and observe that
    \begin{align*}
        \vec{u}\in \nullspace ( X^*Y) 
        &\implies (X^*Y) \vec{u}=\vec{0} \\
        &\implies \vec{u}^*( X^*Y) \vec{u}=0 \\
        &\implies \vec{u}^*( X^*AX) \vec{u}=0 \\
        &\implies ( X\vec{u})^*A( X\vec{u}) =0 \\
        &\implies A( X\vec{u}) =\vec{0} \\
        &\implies ( AX) \vec{u}=\vec{0} \\
        &\implies Y\vec{u}=\vec{0} \\
        &\implies \vec{u}\in \nullspace Y.
    \end{align*}
    
    \medskip\noindent
    ($\Leftarrow $) Let $r=\rank X$. Since $X^*Y$ is Hermitian, we
    have $\nullspace X\subseteq \nullspace (Y^*X)=\nullspace (X^*Y)=\nullspace Y$, so $\nullspace X\subseteq \nullspace Y$. The hypotheses of
    Theorem \ref{Theorem:Hermitian} are satisfied, and its proof shows that
    \begin{equation}\label{eq:PSDB}
        B= 
        \begin{bmatrix}
        B_{1} & B_{2}
        \end{bmatrix}
         = 
        \begin{bmatrix}
        Z_{1}\Sigma_{r}^{-1} & ( Z_{2}\Sigma_{r}^{-1})^* \\ 
        Z_{2}\Sigma_{r}^{-1} & \lambda I_{m-r}
        \end{bmatrix}
         = 
        \begin{bmatrix}
        H & L^* \\ 
        L & \lambda I_{m-r}
        \end{bmatrix}
    \end{equation}
    is Hermitian for any $\lambda >0$. The identity \eqref{eq:BasicX*Y} and the
    assumption $X^*Y\geq 0$ imply that $\Sigma_{r}Z_{1}\geq 0$, so 
    $\Sigma_{r}^{-* }( \Sigma_{r}Z_{1}) \Sigma_{r}^{-1}=Z_{1}\Sigma_{r}^{-1}=H\geq 0$. 
    The hypothesis $\nullspace Y=\nullspace (X^*Y)$, the fact that $H\geq 0$, and \eqref{eq:rankZ} ensure that
    \begin{equation}\label{eq:nullHL}
        \nullspace H=\nullspace B_{1}=\nullspace H\cap \nullspace L\subseteq \nullspace L.  
    \end{equation}
    Let $\lambda_{1} \geq 0$ be the
    largest eigenvalue of the positive semidefinite matrix $LH^{\dagger }L^*$. We claim that $B\geq 0$ if $\lambda \geq \lambda_{1}$.
    Lemma \ref{Lemma:SchurComplement}.c and \eqref{eq:nullHL} ensure that $B$ is $^*$congruent to $H\oplus ( \lambda I_{m-r}-LH^{\dagger }L^*) $, which is
    positive semidefinite if $\lambda I_{m-r}-LH^{\dagger }L^*\geq 0$. If $\lambda \geq \lambda_{1}$ and $\vec{x}\in \F^{m-r}$, then
    \begin{equation*}
        \vec{x}^*( \lambda I_{m-r}-LH^{\dagger }L^*) 
        \vec{x}=\lambda \vec{x}^*\vec{x}-\vec{x}^*(LH^{\dagger }L^*) \vec{x}\geq ( \lambda -\lambda_{1}) \norm{ \vec{x} }^{2}\geq 0,
    \end{equation*}
    which validates our claim.
\end{proof}

\begin{corollary}
    Let $X,Y\in \M_{m\times n}(\F)$ with $m\geq n\geq 1$ and $X\neq 0$.
    There is a positive definite $A\in \M_{m}(\F)$ such that $AX=Y$ if
    and only if $X^*Y\geq 0$, $\nullspace Y=\nullspace (X^*Y)$, and $\rank Y=\rank X$.
\end{corollary}

\begin{proof}
    ($\Rightarrow $) If $A\in \M_{m}(\F)$ is positive definite and $AX=Y$,
    then $A$ is invertible and $\rank X=\rank Y$. Moreover, $\nullspace Y=\nullspace ( X^*Y) $ and $X^*Y\geq 0$, as in the
    preceding theorem.
    
    \medskip\noindent($\Leftarrow $) Let $r=\rank X$. Since $\rank Y=\rank X$ and $\nullspace Y=\nullspace (X^*Y)$, it follows that 
    \begin{equation*}
        r=\rank Y=\rank (X^*Y)=\rank Z_{1}=\rank H.
    \end{equation*}
    Thus, $H\geq 0$ (as in the preceding theorem), $H\in \M_{r}(\F)$ is
    invertible, and $H$ is positive definite. Lemma \ref{Lemma:SchurComplement}.b 
    ensures that $B$ in \eqref{eq:PSDB} is $^*$congruent to $H\oplus( \lambda I_{n-r}-LH^{-1}L^*) $, which is positive definite
    if $\lambda $ is greater than the largest eigenvalue of $LH^{-1}L^{*}$.
\end{proof}

\begin{corollary}
    Let $X,Y\in \M_{m\times n}(\F)$ with $m\geq n\geq 1$ and $\rank X=n$. 
    There is a positive definite $A\in \M_{m}$ such that $AX=Y$ if and only if $X^*Y>0$.
\end{corollary}

\begin{proof}
    ($\Rightarrow $) If $A\in \M_{m}(\F)$ is positive definite and $AX=Y$,
    then $X^*Y=X^*AX$, which is positive definite since $X$ has full rank.
    
    \medskip\noindent($\Leftarrow $) If $X^*Y>0$, then $n=\rank ( X^{*}Y) \leq \rank Y$, 
    so $\rank Y=n$ and $\nullspace Y=\{\vec{0}\}=\nullspace (X^*Y)$. The preceding corollary ensures that
    there is a positive definite $A\in \M_{m}$ such that $AX=Y$.
\end{proof}

\section{Unitary targeting}

The unitary targeting problem generalizes the familiar fact that, for given $\vec{x},\vec{y}\in \F^{n}$, there is a unitary 
$U\in \M_{n}(\F)$ such that $\vec{y}=U\vec{x}$ if and only if $\norm{\vec{x}} =\norm{ \vec{y} }$.

\begin{theorem}
    Let $X,Y\in \M_{m\times n}(\F)$ with $m\geq n\geq 1$ and $X\neq 0$.
    There is a unitary $A\in \M_{m}(\F)$ such that $AX=Y$ if and only if $X^*X=Y^*Y$.
\end{theorem}

\begin{proof}
    ($\Rightarrow $) If $Y=AX$ and $A$ is unitary, then $Y^*Y=(AX)^*( AX) =X^*( A^*A)X=X^*IX=X^*X$.
    
    \medskip\noindent($\Leftarrow $) Let $r=\rank X$. If $X^*X=Y^*Y$, then 
    $\nullspace X=\nullspace Y$ \cite[Theorem 15.1.9]{GH2} and \eqref{eq:V*YW}
    ensures that $B= 
    \begin{bmatrix}
    B_{1} & B_{2}
    \end{bmatrix}
     $ satisfies \eqref{eq:BasicSigma} with $B_{1}=Z\Sigma_{r}^{-1}$. It
    suffices to show that $B_{1}$ has orthonormal columns, since we may then
    choose $B_{2}$ to make $B$ unitary; see \cite[Corollary 6.3.14]{GH2}.
    Observe that
    \begin{equation*}
        W^*( Y^*Y) W
        =( V^*YW)^*(V^*YW) 
        = 
        \tworowvector{Z}{0}^* 
        \tworowvector{Z}{0}
         = 
        \begin{bmatrix}
            Z^*Z & 0 \\ 
            0 & 0
        \end{bmatrix}
         .
    \end{equation*}
    We also have
    \begin{equation*}
        W^*( X^*X) W=\Sigma^*\Sigma = 
        \begin{bmatrix}
            \Sigma_{r}^{2} & 0 \\ 
            0 & 0
        \end{bmatrix}.
    \end{equation*}
    Since $X^*X=Y^*Y$, we conclude that $Z^*Z=\Sigma_{r}^{2}$
    and hence 
    \begin{equation*}
        I_{r}=( Z\Sigma_{r}^{-1})^*( Z\Sigma
        _{r}^{-1}) =B_{1}^*B_{1},
    \end{equation*}
    that is, $B_{1}$ has orthonormal columns.
\end{proof}

The polar decomposition \cite[\S\ 16.3]{GH2} or \cite[\S\ 7.3]{MA2} provides
an alternative approach to the unitary targeting problem. If $X^{*}X=Y^*Y$ and $Q=( X^*X)^{1/2}$, 
then $X=U_{1}Q$ and $Y=V_{1}Q$, in which $U_{1},V_{1}\in \M_{m\times n}(\F)$ have
orthonormal columns. If 
$U= \tworowvector{U_1}{U_2}$ and $V= \tworowvector{V_1}{V_2}$ are unitary, then $VU^*$ is unitary and 
\begin{equation*}
    ( VU^*) X=( VU^*) U_{1}Q=V( U^{*}U_{1}) Q= 
    \tworowvector{V_1}{V_2}
    \begin{bmatrix}
    I \\ 
    0
    \end{bmatrix}
     Q=V_{1}Q=Y.
\end{equation*}
If $X$ and $Y$ do not have the same numbers of columns, there is an
interesting story to be told about the identity $X^*X=Y^*Y$. For
details, examples, and a historical review, see \cite{HO}.

\section{Reflection Targeting}

A matrix $A\in \M_{m}(\F)$ is a \emph{reflection} if it is a
Hermitian involution ($A=A^*$ and $A^{2}=I)$, or, equivalently, if it
is Hermitian and unitary. We now consider the reflection targeting problem.

\begin{theorem}\label{Theorem:Reflection}
    Let $X,Y\in \M_{m\times n}(\F)$ with $m\geq n\geq 1$ and $r=\rank Y\geq 1$. 
    There is a reflection $A\in \M_{m}(\F)$ such that $AX=Y$ if and only if $X^*Y$ is Hermitian and $X^*X=Y^*Y$.
\end{theorem}

\begin{proof}
    ($\Rightarrow $) If $A$ is a reflection and $AX=Y$, then $A$ is Hermitian
    and so is $X^*Y=X^*AX$. Since $A$ is unitary, $Y^*Y=(AX)^*AX=X^*A^*AX=X^*IX=X^*X$. 
    
    \medskip\noindent($\Leftarrow $) Suppose that $X^*Y=Y^*X$ and $X^*X=Y^{*}Y$. 
    Let $E=X+Y$ and $F=X-Y$, so that $X=\frac{1}{2}( E+F) $ and 
    $Y=\frac{1}{2}( E-F) $. Let $P=F( F^*F)^{\dagger}F^*$, which is the (necessarily Hermitian) orthogonal projection onto 
    $\colspace F$. Define the Hermitian matrix $A=I-2P$, which is a reflection since
    \begin{equation*}
        A^{2}=( I-2P)^{2}=I-4P+4P^{2}=I-4P+4P=I.
    \end{equation*}
    Observe that $\colspace E$ is orthogonal to $\colspace F$ since
    \begin{equation*}
        E^*F=( X+Y)^*( X-Y) =( X^{*}X-Y^*Y) +( Y^*X-X^*Y) =0+0=0.
    \end{equation*}
    Therefore, $PE=0$ and $PF=F$. Consequently,
    \begin{align*}
        AX 
        &= \frac{1}{2}( I-2P) ( E+F) =\frac{1}{2}(E+F-2PE-2PF) \\
        &= \frac{1}{2}( E+F-0-2F) =\frac{1}{2}( E-F) =Y. \qedhere
    \end{align*}
\end{proof}

For a given nonzero $Y\in \M_{m\times n}(\F)$, the following
corollary shows how to construct all the $X\in \M_{m\times n}(\F)$
such that $Y=AX$ for some reflection $A$.

\begin{corollary}
    Let $X,Y\in \M_{m\times n}(\F)$ with $m\geq n\geq 1$ and $r=\rank Y\geq 1$. 
    Let $Y=V\Sigma W^*$ be a singular value decomposition, in which 
    \begin{equation*}
        V^*YW=\Sigma = 
        \begin{bmatrix}
            \Sigma_{r} & 0 \\ 
            0 & 0
        \end{bmatrix}
         \in \M_{m\times n}(\R)
    \end{equation*}
    and $\Sigma_{r}\in \M_{r}(\R)$ is diagonal and positive definite.
    There is a reflection $A\in \M_{m}(\F)$ such that $AX=Y$ if and only if
    \begin{equation}\label{eq:Reflect}
        V^*XW= 
        \begin{bmatrix}
            U_{11}\Sigma_{r} & 0 \\ 
            U_{21}\Sigma_{r} & 0
        \end{bmatrix},
    \end{equation}
    in which $U_{11}\in \M_{r}(\F)$ is Hermitian and 
    $U_{1}= \big[\begin{smallmatrix} U_{11} \\ U_{21} \end{smallmatrix} \big] \in \M_{m\times r}(\F)$ has orthonormal columns.
\end{corollary}

\begin{proof}
    ($\Rightarrow $) Partition
    \begin{equation*}
        V^*XW= 
        \begin{bmatrix}
            Z_{11} & Z_{12} \\ 
            Z_{21} & Z_{22}
        \end{bmatrix}
    \end{equation*}
    conformally to $\Sigma $. Since $X^*Y$ is Hermitian, Corollary \ref{Corollary:HermitianZ} 
    ensures that $Z_{12}=0$ and $\Sigma_{r}Z_{11}=Z_{11}^{*}\Sigma_{r}$. Since $X^*X=Y^*Y$, we have
    \begin{align}
        W^*X^*XW 
        &=  
        \begin{bmatrix}
        Z_{11}^* & Z_{21}^* \\ 
        0 & Z_{22}^*
        \end{bmatrix}
        \begin{bmatrix}
        Z_{11} & 0 \\ 
        Z_{21} & Z_{22}
        \end{bmatrix}
          \notag \\
        &=  
        \begin{bmatrix}
        Z_{11}^*Z_{11}+Z_{21}^*Z_{21} & Z_{21}^*Z_{22} \\ 
        Z_{22}^*Z_{21} & Z_{22}^*Z_{22}
        \end{bmatrix}
         = 
        \begin{bmatrix}
        \Sigma_{r}^{2} & 0 \\ 
        0 & 0
        \end{bmatrix}
          \label{eq:Zcon2} \\
        &= W^*Y^*YW.  \notag
    \end{align}
    A comparison of entries in \eqref{eq:Zcon2} reveals that $Z_{22}^{*}Z_{22}=0$, and hence $Z_{22}=0$. If we partition
    \begin{equation*}
    V^*XW= 
    \tworowvector{Z_1}{0},
    \end{equation*}
    another comparison of entries shows that $Z_{1}^*Z_{1}=\Sigma_{r}^{2}$.  
    It follows from the polar decomposition that $Z_{1}=U_{1}\Sigma_{r}$, in which 
    $U_{1}= \big[ \begin{smallmatrix} U_{11} \\  U_{21} \end{smallmatrix} \big] \in \M_{m\times r}(\F)$ has orthonormal columns. The
    condition $\Sigma_{r}Z_{11}=Z_{11}^*\Sigma_{r}$ ensures that 
    $\Sigma_{r}U_{11}\Sigma_{r}=\Sigma_{r}U_{11}^*\Sigma_{r}$, so $U_{11}$ is Hermitian. 
    
    \medskip\noindent($\Leftarrow $) If $V^*XW$ can be partitioned as in \eqref{eq:Reflect},
    then
    \begin{align*}
        W^*( X^*X) W &=  
        \begin{bmatrix}
        \Sigma_{r}U_{1}^* \\ 
        0
        \end{bmatrix}
        \begin{bmatrix}
        U_{1}\Sigma_{r} & 0
        \end{bmatrix}     
        =  
        \begin{bmatrix}
        \Sigma_{r}U_{1}^*U_{1}\Sigma_{r} & 0 \\ 
        0 & 0
        \end{bmatrix}
        \\
         &= 
        \begin{bmatrix}
        \Sigma_{r}^{2} & 0 \\ 
        0 & 0
        \end{bmatrix}
         =W^*( Y^*Y) W,
    \end{align*}
    which shows that $X^*X=Y^*Y$. Also,
    \begin{align*}
        W^*( X^*Y) W 
        &=  
        \begin{bmatrix}
        \Sigma_{r}U_{11}^* & \Sigma_{r}U_{21}^* \\ 
        0 & 0
        \end{bmatrix}  
        \begin{bmatrix}
        \Sigma_{r} & 0 \\ 
        0 & 0
        \end{bmatrix}
         \\
        &=  
        \begin{bmatrix}
        \Sigma_{r}U_{11}^*\Sigma_{r} & 0 \\ 
        0 & 0
        \end{bmatrix}
         = 
        \begin{bmatrix}
        \Sigma_{r}U_{11}\Sigma_{r} & 0 \\ 
        0 & 0
        \end{bmatrix}
    \end{align*}
    is Hermitian, which shows that $X^*Y$ is Hermitian. The preceding
    theorem now ensures that there is a reflection $A\in \M_{m}(\F)$ such that $AX=Y$.
\end{proof}

One consequence of Theorem \ref{Theorem:Reflection} is that, for given 
$\vec{x},\vec{y}\in \C^{n}$, there is a reflection $A\in \M_{n}(\C)$ such that $A\vec{x}=\vec{y}$ if and only if 
$\norm{\vec{x}} =\norm{ \vec{y} }$ and $\vec{x}^*\vec{y}$ is real; the latter requirement is superfluous if 
$\F=\R$. In fact, $A$ may be chosen to be a scalar multiple of a Householder matrix in this case. 
Our proof of Theorem \ref{Theorem:Reflection} constructs a reflection matrix that is a natural analog
of a Householder matrix.

\section{Orthogonal Projection Targeting}

A computation with the Moore--Penrose pseudoinverse provides a necessary and
sufficient condition for there to be a solution to the orthogonal projection targeting problem.

\begin{theorem}\label{Theorem:OrthogProj}
    Let $X,Y\in \M_{m\times n}(\F)$ with $m\geq n\geq 1$ and $Y\neq 0$. There is an orthogonal projection $A\in \M_{m}(\F)$
    such that $AX=Y$ if and only if $Y^*X=Y^*Y$.
\end{theorem}

\begin{proof}
    ($\Rightarrow $) If $Y=AX$, $A$ is Hermitian, and $A^{2}=A$, then 
    \begin{equation*}
        Y^*Y=( AX)^*( AX) =X^*A^{*}AX=X^*A^{2}X=X^*( AX) =X^*Y.
    \end{equation*}

    \medskip\noindent($\Leftarrow $) If $Y^*X=Y^*Y$, let $A=Y( Y^*Y)^{\dagger }Y^*$. Then $A$ is Hermitian and
    \begin{equation*}
        A^{2}=Y( Y^*Y)^{\dagger }( Y^*Y) (Y^*Y)^{\dagger }Y^*=Y( Y^*Y)^{\dagger}Y^*=A,
    \end{equation*}
    so $A$ is an orthogonal projection. Since $( Y^*Y)^{\dagger }( Y^*Y)$ 
    is the orthogonal projection on $\colspace( Y^*Y)^*=\colspace ( Y^*Y) =\colspace Y^*$, it follows that 
    $( Y^*Y)^{\dagger}( Y^*Y) Y^*=Y^*$ and hence $Y=Y( Y^{*}Y) ( Y^*Y)^{\dagger }$. Therefore,
    \begin{equation*}
        AX=Y( Y^*Y)^{\dagger }( Y^*X) =Y(Y^*Y)^{\dagger }( Y^*Y) =Y. \qedhere
    \end{equation*}
\end{proof}

We can use a singular value decomposition of $Y$ (note the change of strategy here) to identify all solutions $X$ to
the orthogonal projection targeting problem $AX=Y$ for a given $Y$.

\begin{corollary}
    Let $X,Y\in \M_{m\times n}(\F)$ with $m\geq n\geq 1$ and $r=\rank Y\geq 1$. Let $Y=V\Sigma W^*$ be a singular value decomposition, in which 
    \begin{equation*}
        V^*YW=\Sigma = 
        \begin{bmatrix}
            \Sigma_{r} & 0 \\ 
            0 & 0
        \end{bmatrix}
         \in \M_{m\times n}(\R)
    \end{equation*}
    and $\Sigma_{r}\in \M_{r}(\R)$ is diagonal and positive definite. Partition
    \begin{equation*}
        V^*XW=Z= 
        \begin{bmatrix}
            Z_{11} & Z_{12} \\ 
            Z_{21} & Z_{22}
        \end{bmatrix}
    \end{equation*}
    conformally to $\Sigma $. There is an orthogonal projection $A\in \M_{m}(\F)$ such that $AX=Y$ if and only if
    \begin{equation*}
        V^*XW= 
        \begin{bmatrix}
            \Sigma_{r} & 0 \\ 
            Z_{21} & Z_{22}
        \end{bmatrix},
    \end{equation*}
    in which $Z_{21}\in \M_{(m-r)\times r}(\F)$ and $Z_{22}\in \M_{(m-r)\times (n-r)}(\F)$.
\end{corollary}

\begin{proof}
    Compare entries in
    \begin{align*}
        Y^*Y=Y^*X 
        &\iff W\Sigma^{\T}\Sigma W^*=W\Sigma^{\T}V^*VZW^* \\
        &\iff \Sigma^{\T}\Sigma =\Sigma^{\T}Z \\
        &\iff  
        \begin{bmatrix}
        \Sigma_{r}^{2} & 0 \\ 
        0 & 0
        \end{bmatrix}
         = 
        \begin{bmatrix}
            \Sigma_{r}Z_{11} & \Sigma_{r}Z_{12} \\ 
            0 & 0
        \end{bmatrix}
         . \qedhere
    \end{align*}
\end{proof}

\section{Complex symmetric targeting}

A real symmetric matrix is Hermitian, but a complex symmetric matrix need
not even be normal. However, there is an analog of Theorem \ref{Theorem:Hermitian}
for complex symmetric targeting.

\begin{theorem}
    Let $X,Y\in \M_{m\times n}(\C)$ with $m\geq n\geq 1$ and $X\neq 0$.
    There is a complex symmetric $A\in \M_{m}(\C)$ such that $AX=Y$ if
    and only if $\nullspace X\subseteq \nullspace Y$ and $X^{\T}Y$ is complex symmetric.
\end{theorem}

\begin{proof}
    ($\Rightarrow $) If $Y=AX$ and $A$ is complex symmetric, then 
    $\nullspace X\subseteq \nullspace Y$, $X^{\T}AX$ is complex symmetric, and $X^{\T}Y=X^{\T}AX=( X^{\T}Y)^{\T}$.

    \medskip\noindent($\Leftarrow $) Let $\rank X=r\geq 1$ and adopt the notation in the
    proof of Theorem \ref{Theorem:Basic}. Let $X=V\Sigma W^*$ be a singular value
    decomposition. If there is a complex symmetric $F\in \M_{m}(\C)$ such that $F\Sigma =V^{\T}YW$, then 
    \begin{equation*}
        \overline{V}( F\Sigma ) W^*=\overline{V}( V^{\T}YW) W^{*}=Y
    \end{equation*}
    and 
    \begin{equation*}
        Y=\overline{V}( F\Sigma ) W^*=( \overline{V}FV^*)( V\Sigma W^*) =( \overline{V}FV^*) X.
    \end{equation*}
    Thus, $A=\overline{V}FV^*$ is complex symmetric and $AX=Y$. 
    If $\nullspace X\subseteq \nullspace Y$, then $YW_{2}=0$ and 
    \begin{equation*}
        V^{\T}YW= 
        \begin{bmatrix}
            V^{\T}YW_{1} & 0
        \end{bmatrix}.
    \end{equation*}
    Partition $F= \tworowvector{F_1}{F_2}$ with 
    \begin{equation*}
        F_{1}= 
        \begin{bmatrix}
            V_{1}^{\T}YW_{1}\Sigma_{r}^{-1} \\[3pt] 
            V_{2}^{\T}YW_{1}\Sigma_{r}^{-1}
        \end{bmatrix}
         \in \M_{m\times r}(\C).
    \end{equation*}
    For any choice of $F_{2}\in \M_{m\times (m-r)}(\C)$, we have $F\Sigma=V^{\T}YW$. Now compute
    \begin{align*}
        W^{\T}( X^{\T}Y) W 
        &= W^{\T}( ( V\Sigma W^*)^{\T}Y) W=W^{\T}( \overline{W}\Sigma^{\T}V^{\T}Y) W \\
        &= \Sigma^{\T}( V^{\T}YW) =\Sigma^{\T} 
        \begin{bmatrix}
        V^{\T}YW_{1} & 0
        \end{bmatrix}\\
        &=  
        \begin{bmatrix}
            \Sigma_{r} & 0 \\ 
            0 & 0
        \end{bmatrix}  
        \begin{bmatrix}
            V_{1}^{\T}YW_{1} & 0 \\[3pt] 
            V_{2}^{\T}YW_{1} & 0
        \end{bmatrix}
         = 
        \begin{bmatrix}
            \Sigma_{r}V_{1}^{\T}YW_{1} & 0 \\ 
            0 & 0
        \end{bmatrix}.
    \end{align*}
    If $X^{\T}Y$ is symmetric, then so are $\Sigma_{r}V_{1}^{\T}YW_{1}$ and 
    $\Sigma_{r}^{-1}( \Sigma_{r}V_{1}^{\T}YW_{1}) \Sigma_{r}^{-1}=V_{1}^{\T}YW_{1}\Sigma_{r}^{-1}$. 
    If $G\in \M_{m-r}$ is any complex symmetric matrix, then
    \begin{equation*}
        F= 
        \begin{bmatrix}
            V_{1}^{\T}YW_{1}\Sigma_{r}^{-1} & ( V_{2}^{\T}YW_{1}\Sigma_{r}^{-1})^{\T} \\[3pt] 
            V_{2}^{\T}YW_{1}\Sigma_{r}^{-1} & G
        \end{bmatrix}
    \end{equation*}
    is complex symmetric and $F\Sigma = 
    \begin{bmatrix}
        V^{\T}YW_{1} & 0
    \end{bmatrix}
     =V^{\T}YW$. Thus, $A=\overline{V}FV^*$ is complex symmetric and $AX=Y$.
\end{proof}

\section{Normal Targeting}

What about the normal targeting problem? If $n=1$ and $\vec{x},\vec{y}\in \F^{m}$ are nonzero, then 
$\vec{x}/\norm{ \vec{x} }$ and $\vec{y}/\norm{ \vec{y} }$
have the same norm, so there is a unitary $U\in \M_{m}(\F)$ such that 
$U( \vec{x}/\norm{ \vec{x} }) =\vec{y} / \norm{ \vec{y}}$. Then 
$A=( \norm{ \vec{y}}/\norm{ \vec{x}}) U$ is normal and $A\vec{x}=\vec{y}$.

For $n\geq 2$, however, it is not (yet) clear how to proceed. For a normal $A \in \M_m(\F)$, 
we have $A^*\!A=AA^*$, which is a system of $m^{2}$ equations
in the entries of $A$. The equations corresponding to diagonal entries of $A^*\! A$ and $AA^*$ say that, for each $i$, 
the $i$th row and the $i$th column of $A$ have the same Euclidean norm. 
If we partition $A=\big[ \begin{smallmatrix} B & D \\ C & E \end{smallmatrix}\big] \in \M_{m}(\F)$, in which $B\in \M_{k}(\F)$, then
inspection of the $(1,1)$ block of $A^*A=AA^*$ reveals that $B^*B-BB^*+C^*C=DD^*$. Consequently, for given $B$
and $C$, in order for there to be matrices $D$ and $E$ such that $A$ is normal, it is necessary that
\begin{equation}\label{eq:HBC}
H(B,C)=B^*B-BB^*+C^*C  
\end{equation}%
be positive semidefinite. If $m=2k$, this necessary condition is sufficient
for $k=2$ (see \cite{FI1}), but not for $k\geq 3$. Consider
\begin{equation*}
B=
\begin{bmatrix}
0 & 0 & 1 \\ 
1 & 0 & 0 \\ 
0 & 1 & 0%
\end{bmatrix}
 \quad \text{and}\quad
 C=
\begin{bmatrix}
1 & 0 & 0 \\ 
0 & 0 & 0 \\ 
0 & 0 & 0%
\end{bmatrix}.
\end{equation*}%
In this case, $B$ is normal and $H(B,C)=C$ is positive semidefinite.
However, there is no choice of $D$ and $E$ that makes $A$ normal because the
first row and first column of $A$ cannot have the same Euclidean norm. For a
discussion of the condition \eqref{eq:HBC} see \cite{FI1} and \cite{FI2}.

Although the normal targeting problem remains open, the following theorem provides a solution in a special case that
provides an alternative solution to the reflection and orthogonal projection
targeting problems.

\begin{theorem}\label{Theorem:Normal2Eval}
Let $X,Y\in \M_{m\times n}(\C)$  with $m \geq n \geq 1$ and $X \neq 0$.  Let $\lambda,\mu \in \C$ be distinct. 
If $m=n$ and either $Y = \lambda X$ or $Y = \mu X$, assume that $\rank X < m$.
\begin{enumerate}[leftmargin=*]
    \item There is a normal $A\in \M_{m}(\C)$ such that the spectrum of $A$ is contained in $\{\lambda ,\mu \}$ and $AX=Y$ if and only if $( Y-\lambda X)^*( Y-\mu X) =0$. 
    \item If $X$, $Y$, $\lambda $, and $\mu $ are real and $( Y-\lambda X)^{\T}( Y-\mu X) =0$, then there is a real normal 
    $A\in\M_{m}(\R)$ such that the spectrum of $A$ is contained in $\{\lambda ,\mu \}$ and $AX=Y$.
\end{enumerate}
\end{theorem}

\begin{proof}
(a) $( \Rightarrow ) $ If $AX=Y$, $A$ is normal, and the
eigenvalues of $A$ are $\lambda $ (with multiplicity $p\geq 1)$ and $\mu$
(with multiplicity $m-p\geq 1)$, then there is a unitary $U\in \M_{m}(\C)$ 
and a diagonal $\Lambda =\lambda I_{p}\oplus \mu I_{m-p}$ such that $A=U\Lambda U^*$. Then
\begin{align*}
    ( Y-\lambda X)^*( Y-\mu X) 
    &= ( AX-\lambda X)^*( AX-\mu X) \\
    &= X^*( A-\lambda I)^*( A-\mu I) X \\
    &= X^*U( \overline{\Lambda}-\overline{\lambda}I) U^*U(\Lambda -\mu I) U^*X \\
    &= X^*U 
    \begin{bmatrix}
        0_{p} & 0 \\ 
        0 & ( \overline{\mu}-\overline{\lambda}) I_{m-p}
    \end{bmatrix}
    \begin{bmatrix}
        ( \lambda -\mu ) I_{p} & 0 \\ 
        0 & 0_{m-p}
    \end{bmatrix}
     U^*X \\
    &= X^*U\,0_{m}U^*X=0.
\end{align*}

\medskip\noindent($\Leftarrow $) Suppose that $( Y-\lambda X)^*( Y-\mu X) =0$. 
Let $E=Y-\mu X$ and $F=Y-\lambda X$ and observe that $F^{*}E=0=E^*F$, 
so $\colspace E$ and $\colspace F$ are orthogonal subspaces. Moreover,
\begin{equation*}
    X=( \lambda -\mu )^{-1}( E-F) \quad \text{and}\quad  Y=( \lambda -\mu )^{-1}( \lambda E-\mu F) .
\end{equation*}
Let $P=E(E^*E)^{\dagger }E^*$ and $Q=F(F^*F)^{\dagger}F^*$, 
which are, respectively, orthogonal projections onto the
respective orthogonal subspaces $\colspace E$ and $\colspace F$. 
If $E = 0$ (respectively, $F = 0$), the hypotheses ensure that $\rank F < m$ (respectively, $\rank E < m$).
If $\rank E + \rank F < m$, let $R$ be the orthogonal projection onto $(\colspace E \oplus \colspace F)^{\perp}$;
otherwise, let $R = 0_m$. Let 
$A=\lambda P+\mu Q + \lambda R$, which is normal with spectrum $\{\lambda ,\mu \}$ \cite[\S\ 14.9]{GH2}. Then 
\begin{align*}
    AE &= ( \lambda P+\mu Q + \lambda R) E=\lambda PE+\mu QE + \lambda RE=\lambda PE=\lambda E, \\
    AF &= ( \lambda P+\mu Q  + \lambda R) F=\lambda PF+\mu QF + \lambda RF=\mu QF=\mu F,
\end{align*}
and
\begin{equation*}
    AX=( \lambda -\mu )^{-1}A( E-F) =( \lambda -\mu)^{-1}( AE-AF) =( \lambda -\mu )^{-1}(\lambda E-\mu F) =Y.
\end{equation*}

\medskip\noindent (b) If $X$, $Y$, $\lambda $, and $\mu $ are real, then the construction in
part (a) creates real matrices $E$, $F$, $P$, $Q$, and $A$.
\end{proof}

The assumption about $\rank X$ cannot be omitted from the preceding theorem.  If $X \in \M_m(\C)$ is invertible
and $Y = \lambda X$, then the unique solution to the targeting problem is $A = \lambda I$, which has only one point in its spectrum.

Let $X,Y\in \M_{m\times n}(\C)$. An orthogonal projection is a normal
matrix with spectrum contained in $\{0,1\}$. Theorem \ref{Theorem:Normal2Eval} says that there is
an orthogonal projection $A$ such that $AX=Y$ if and only if $(Y-X)^*( Y-0X) =Y^*Y-X^*Y=0$. This is the
condition in Theorem \ref{Theorem:OrthogProj}.

A reflection is a normal matrix with spectrum contained in $\{1,-1\}$. Theorem \ref{Theorem:Normal2Eval} 
says that there is a reflection $A$ such that $AX=Y$ if and
only if $C=( Y+X)^*( Y-X) =0$. Observe that $C=( Y^*Y-X^*X) +( X^*Y-Y^*X) $
is the sum of a Hermitian matrix and a skew-Hermitian matrix, and that $C=0$
if and only if both its Hermitian part and its skew-Hermitian part are zero.
This is the condition in Theorem \ref{Theorem:Reflection}.

\medskip\noindent\textbf{Acknowledgement.} We thank Michael A.~Dritschel for pointing out that \cite[Corollary 1]{ZS}
gives a version of Theorem \ref{Theorem:PSD} that is valid for bounded linear
operators on a Hilbert space.  We thank the anonymous referee for detailed comments on the initial draft.

\bibliography{LinearTargetingProblem}

\providecommand{\bysame}{\leavevmode\hbox to3em{\hrulefill}\thinspace}
\providecommand{\MR}{\relax\ifhmode\unskip\space\fi MR }
\providecommand{\MRhref}[2]{%
  \href{http://www.ams.org/mathscinet-getitem?mr=#1}{#2}
}
\providecommand{\href}[2]{#2}
\begin{thebibliography}{1}

\bibitem{FI1}
Stephen~H. Friedberg and Arnold~J. Insel, \emph{Hyponormal {$2\times 2$}
  matrices are subnormal}, Linear Algebra Appl. \textbf{175} (1992), 31--38.
  \MR{1179339}

\bibitem{FI2}
\bysame, \emph{Characterizations of subnormal matrices}, Linear Algebra Appl.
  \textbf{231} (1995), 1--13. \MR{1361098}

\bibitem{GH2}
Stephan~Ramon Garcia and Roger~A. Horn, \emph{Matrix mathematics---a second
  course in linear algebra}, second ed., Cambridge Mathematical Textbooks,
  Cambridge University Press, Cambridge, 2023. \MR{4574833}

\bibitem{MA2}
Roger~A. Horn and Charles~R. Johnson, \emph{Matrix analysis}, second ed.,
  Cambridge University Press, Cambridge, 2013. \MR{2978290}

\bibitem{HO}
Roger~A. Horn and Ingram Olkin, \emph{When does {$A^*A=B^*B$} and why does one
  want to know?}, Amer. Math. Monthly \textbf{103} (1996), no.~6, 470--482.
  \MR{1390576}

\bibitem{ZS}
Zolt\'{a}n Sebesty\'{e}n, \emph{Restrictions of positive operators}, Acta Sci.
  Math. (Szeged) \textbf{46} (1983), no.~1-4, 299--301. \MR{739047}

\end{thebibliography}
\bibliographystyle{amsplain}

\end{document}